\documentclass[a4paper,oneside,10pt]{article}
\usepackage{amsthm,amssymb}
\usepackage[a4paper,margin=3cm,centering,nohead,ignorefoot,footskip=5ex]{geometry}
\usepackage{latexsym}
\usepackage{stmaryrd}
\usepackage{enumitem, hyperref}
\makeatletter
\def\namedlabel#1#2{\begingroup
    #2%
    \def\@currentlabel{#2}%
    \phantomsection\label{#1}\endgroup
}
\makeatother
\usepackage[leqno]{amsmath}
\usepackage{multicol}
\makeatletter
\newcommand{\leqnomode}{\tagsleft@true\let\veqno\@@leqno}

\newtheorem{definition}{Definition}[section]
\newtheorem{proposition}[definition]{Proposition}
\newtheorem{lemma}[definition]{Lemma}
\newtheorem{theorem}[definition]{Theorem}
\newtheorem{remark}[definition]{Remark}
\newtheorem{corollary}[definition]{Corollary}

\newtheorem{example}[definition]{Example}
\newtheorem{counterexample}[definition]{Counterexample}
\newtheorem{examples}[definition]{Examples}

\newcommand*{\st}{\mathrm{st}}
\newcommand*{\Ker}{\mathrm{Ker}}

\newenvironment{psmallmatrix}
  {\left(\begin{smallmatrix}}
  {\end{smallmatrix}\right)}

\begin{document} 

\title {Assemblies as Semigroups}

\author{Ulderico Dardano\footnote{Department of Mathematics and Applications, Universit\`{a} di Napoli "Federico II", ulderico.dardano@unina.it}\ , \ Bruno Dinis\footnote{Departamento de Matemática, Universidade de \'{E}vora, bruno.dinis@uevora.pt}\  {\small} and\ \
Giuseppina Terzo\footnote{Department of Mathematics and Applications, Universit\`{a} di Napoli "Federico II", giuseppina.terzo@unina.it}}
\date{}
\maketitle

\vskip-2cm


\maketitle

\begin{abstract}  In this paper we give an algebraic characterization of assemblies in terms of bands of groups. We also consider substructures and homomorphisms of assemblies. We give many examples and counterexamples.
 \end{abstract}

\section{Introduction}

The notion of assembly was introduced in \cite{DinisBerg(11)}, as a generalisation of the notion of group. The main idea is that every element $x$ of an assembly has its own ``neutral'' element denoted $e(x)$ which can be seen as a sort of error term, or the degree of flexibility of the element $x$. Some basic properties of assemblies can be found in \cite{DinisBerg(17),DinisBerg(ta),DinisBerg(19)}. In \cite{DinisBerg(11)} it was also shown that besides groups (which are assemblies for which the function $e$ is always constantly equal to the universal unique neutral element), the so-called {\em external numbers} also satisfy the assembly axioms. 

In order to better understand this sort of algebraic structures and to compare them with existing structures based on semigroups,  we consider a slightly modified definition in which commutativity is not required.  An additional novelty lies in the last condition in Definition \ref{defassembley}. Prior to this, the requirement was that for all $x$ and $y$, $e(xy)$ would have to be equal to either $e(x)$ or $e(y)$ (we sometimes call this a {\em strong assembly}). We now assume a weaker condition which is indeed implied by the previous one: we require $e(xy)=e(x)e(y)$, for all $x$ and $y$, which means that the function of local neutral elements is an homomorphism of semigroups. Thus, by changing the definition, we are proposing a more general notion that, for example,  allows us to consider  the not necessarily commutative assembly $\cal A (G)$ of all cosets  $gN$ where $g$ and $N$ range in a group $G$ and in the lattice $n(G)$ of all normal subgroups of $G$, respectively (see  Example \ref{examples}). Moreover, some expected results such as the fact that any Cartesian product of  assemblies is an assembly now  hold, while before this was  not the case (see Example \ref{examples}).

In this paper we give a characterization of assemblies from a purely algebraic point of view by showing that {\em a semigroup  is an assembly if and only if it is a band of groups and even a semilattice of groups if idempotent elements commute among themselves}. This provides perhaps a new way to approach the theory of bands of groups. For example, we can state that {\em an assembly is strong if and only if the set of idempotents (which are called {\em magnitudes} in the literature) is totally ordered} by the usual relation $x\le y$ if and only if $xy=y$. 
We also briefly consider subassemblies, i.e.\ substructures which are also assemblies, and homomorphisms of assemblies, i.e.\ homomorphisms of semigroups which are assemblies. 

For fundamental results and/or undefined notions about semigroups we refer to 
\cite{Howie} and  \cite{PetrichIntro}. According to what is customary in semigroup theory we generally (but not always)  use multiplicative notation (and consequently juxtaposition) for the binary operation. 


\section{Assemblies as bands of groups}
We introduce the definition of assembly in multiplicative notation. By a semigroup we mean a nonempty set with 
a binary associative operation.

\begin{definition} \label{defassembley}
A nonempty semigroup $(S, \cdot)$ is called an \emph{assembly} if the following hold
\begin{description}[style=multiline, labelwidth=0.8cm]
\item[\namedlabel{assemblyneut}{$(A_1)$}] $\forall x\, \exists e = e(x)\, (xe=ex=x  \wedge \forall f\ (xf=fx=x\to ef=fe=e))$
\item[\namedlabel{assemblysim}{$(A_2)$}]  $\forall x\, \exists s=s(x)\, (xs=sx=e \wedge e(s)=e(x))$
\item [\namedlabel{assemblye(xy)}{$(A_3)$}] $\forall x\,\forall y\,(e(xy)=e(x)e(y))$
\end{description}
To make explicit the functions that exist by conditions \ref{assemblyneut} and \ref{assemblysim} we write $(S,\cdot,e,s)$ instead of $(S,\cdot)$. 
\end{definition}  

The functional notation $e(x)$ and $s(x)$ used above is justified by the fact that the element $e$ and $s$  are unique as in fact:
\begin{itemize}
\item[-]  if $e^{\prime}$ satisfies condition~\ref{assemblyneut}, one has $e^{\prime}=e^{\prime}e=ee^{\prime}=e$,		
\item[-] if $s^{\prime}$ satisfies
condition~\ref{assemblysim} one has $s^{\prime}=s^{\prime}e(s^{\prime}%
)=s^{\prime}e(x)=s^{\prime}xs=xs^{\prime}s=e(x)s=e(s)s=s$.
\end{itemize}

So, we may write indiscriminately $e(x)$ or $x^0$ to denote the unique element $e$ associated with $x$ and $s(x)$ or $x^{-1}$ to denote the unique element element $s$ such that $sx=xs=x^0$.

\begin{examples}\label{examples}\phantom{x}\par\noindent
{
\begin{enumerate}
\item Every (possibly non-commutative) group $G$ is an assembly with $e(x)$ constantly equal to the neutral element of $G$.   
 Furthermore, the semigroup $G\cup\{0\}$, obtained by adding a zero to $G$ (in the usual way, by postulating $0x=x0=0$), is also an assembly. In particular, the multiplicative semigroup of a (possibly skew) field is an assembly, while on the other hand the usual  multiplicative semigroup of the integers is not. 

\item An element $e$ of a semigroup $S$ is said \emph{idempotent} if $e^2=e$.  
 A semigroup in which all elements are idempotent is called a \emph{band} and is clearly an assembly. A group with more than $1$ element is not a band.  
A commutative band is called a {\em semilattice} since  it may be regarded as a 
 lower semilattice with meet operation equal to the product. Clearly, for any set $S$, its powerset has two canonical semigroup structures, $(P(S),\cap)$ and  $(P(S),\cup)$, which are both semilattices.

\item Any totally ordered set is a strong assembly with $xy = \inf\{x, y\}$, and  $e(x) = x = s(x)$. In particular,  $(B,\cup )$ where $B$ is the set of all ordinals less than a given ordinal is a strong assembly.
 
\item  The cartesian product of any family of assemblies is an assembly, with respect to the pointwise multiplication (the proof is a straightforward verification).
 However, the product  of strong assemblies may be not strong as in the example $\{0,1\}\times \{0,1\}$,

\item  The structures $(\mathbb{E},\mathbb{+)}$ and $(\mathbb{E}\backslash \mathcal{N},\cdot)$ are strong assemblies, where $\mathbb{E}$ denotes the external set of external numbers and $\mathcal{N}$ the external set of all neutrices (see \cite[Thm.~4.10]{DinisBerg(11)}).

\item Let $\mathbb{F}$ be a non-archimedean ordered field. Let $C$ be the set of all convex
subgroups for addition of $\mathbb{F}$ and $\mathcal{Q}$ be the set of all cosets with respect to the elements of $C$. The set $\mathcal{Q}$ is called the \emph{quotient class} of $\mathbb{F}$ with respect to $C$. In \cite{DinisBerg(17)} it was show that $\mathcal{Q}$ is a strong assembly.
\end{enumerate}
}
\end{examples}

\noindent By using a standard technique, let us rewrite the assembly axioms with details and proofs.

\begin{lemma}\label{lemma1}
Condition \ref{assemblyneut} is equivalent to 
\begin{equation}\tag{$A_1'$}\label{A1'} 
\forall x \, \exists! e=e(x)\, ( (xe=ex=x) \wedge e^2=e).\hfill
\end{equation}
 So, in particular the set $e(S)=\{e(a)\ |\ a\in S\}$ of all \emph{magnitudes} of an assembly coincides with the set of idempotents of $S$, usally denoted by $E(S)$.
\end{lemma}
{\small 
\begin{proof} If \ref{assemblyneut} holds, $e=e(x)$ is unique by the above and   $ee=e$;  therefore \eqref{A1'} holds.
Conversely if \eqref{A1'} holds and $xf=fx=x$, then we have $f=e$ by unicity in \eqref{A1'}. Hence $e^2=ef=fe=e$ and thus \ref{assemblyneut} holds.
\end{proof}
}

Note that is trivial that if the idempotents commute, i.e.\ $\forall x,y\in E(S) \, (xy=yx$), then $E(S)$ is a semigroup, even a semilattice, as $(xy)(xy)=x(yx)y=x(yx)y=x(xy)y=(xx)(yy)=xy$. But this is not always the case.

\begin{example}  If 
 $A=\begin{psmallmatrix}1 & 0\\ 0& 0\end{psmallmatrix}$ and $M=\begin{psmallmatrix}0 & 1\\ 0& 1\end{psmallmatrix}$, then $S = \{0, A, M, AM\}$ is a semigroup under the usual matrix multiplication, where $MA=0$ and $E(S) = \{0, A, M\}$, but this set is not closed under multiplication. 
  \end{example}

\begin{lemma}\label{A3} Assume that in a semigroup $S$ condition \ref{assemblyneut} holds. Then condition \ref{assemblysim} is equivalent to
\begin{equation}\tag{$A_2'$} \label{A3'}
 \forall x\, \exists! s=s(x)\, ( xs=sx=e(x)).
\end{equation}
\end{lemma}
{\small 
\begin{proof}
 If condition \ref{assemblysim} holds, then it is clear that  the same $s$ works in \eqref{A3'} as well. Let us verify that it is unique. If $s'$ is in the same condition w.r.t. $x$, then by \ref{assemblysim} we have that $e(s)=e(x)=e(s')$, hence 
$s'=s'e(s')=s'e(x)=s'xs=e(s')s=e(s)s=s$, as desired.
Conversely, from \eqref{A3'} it follows immediately    that 
$s(s(x))=x$ and $e(x)=e(s)$ by the symmetric roles of $x$ and $s$.
\end{proof}
}
The first two conditions in the definition of assembly may now be regarded from a different point of view by the next proposition which deals with associative {\em union of groups}. They are also known as {\em completely regular semigroups}  and have been studied in many papers. For the fundamental results on the topic we refer to \cite[ch. IV]{Howie} and   \cite[ch. IV]{PetrichIntro}.  However, here we  prefer to give a short direct proof which uses the Clifford decomposition argument to see a  completely regular semigroup as a union of groups.

\begin{proposition}\label{preAssembly}
For a semigroup $S$ the following are equivalent:
\begin{enumerate}
\item Conditions $(A_1')$ and $(A_2')$ hold;
\item $S=\bigcup_{e\in E(S)} S_e$,  where $S_e=\{a \in S\ |\ e(a)=e\}$ is a group (said Clifford component of $S$ at $e$);
\item  $S$ is a union of (disjoint) groups.
\end{enumerate}
\end{proposition}
{\small
\begin{proof} First of all recall that $(A_1')$ and $(A_1) $ are equivalent and the same happens to $(A_2')$ and $(A_2)$. If all of them hold,  since $e^2=e$ by $(A_2')$ we  have $e=e(e)$ and then $e\in S_e$. Moreover, if $a\in S_e$ then $s(a)\in S_e$ by $(A_2)$. Finally, if $a,b\in S_e$, then $abe=aeb=eab$ and so $ab\in S_e$. Thus each $S_e$ is a group with neutral element $e$. Finally, it is clear that $a\in S_{e(a)}$ for each $a\in S$. Furthermore,  all $S_e$ are disjoint since such are distinct subgroups in any semigroup since groups have only one 
neutral element (compare also to $(A_1')$). 

Finally, it is clear that if $S$ is a union of (disjoint) groups then $(A_1')$ and $(A_2')$ hold by considering, for each $a\in S$, the neutral element and the inverse (resp.) from the unique group in which $a$ lies.
 \end{proof}
}
If we denote by $s(x)=x^{-1}$ and $e(x)=x^0$,  then we have the formulas 
\begin{equation}\tag{$F$}\label{formulas} 
x^0=x^{-1}x=x^{-1}x\ \ \ , \ \ (x^{-1})^{-1}=x\ \ \ , \ \ (xy)^{-1}=y^{-1}x^{-1}
\end{equation}
 which are consistent with the language of group theory and appear also in \cite{DinisBerg(11)}, but in a commutative context. 

Thus, it seems natural  to ask if also the formula 
\begin{equation}\tag{$A_3'$}\label{A_3'} 
(xy)^{0}=x^{0}y^{0}
\end{equation} holds. In other words, we investigate if 
$(A_1)$ and $(A_2)$ imply $(A_3)$, that is if the function $e(x)$ is an  homomorphism. This is certainly {\em true when $A$ is commutative},  as indeed we have $(xy)(x^0y^0)=
x(x^0y)y^0=x(yx^0)y^0)=
xx^0yy^0=xy$ and similarly $(x^0y^0)(xy)=xy$. Moreover, since here we have not used plain commutativity but just the fact that idempotents commute with all other elements, then, according to  \cite[Lemma 3.1]{Clifford41}, we know that the {\em formula $(A_3')$ always holds if idempotents commute}. So we are in a condition to state a consequence of the fundamental Clifford's Theorem \cite[Theorem 3]{Clifford41}.

\medskip

\begin{theorem}\label{cliffordlike}
 Semigroups which are a union of groups and in which idempotents commute are assemblies.
\end{theorem}

Of course there exist elementary semigroups in which idempotents do not commute.

\begin{example}
Any left-zero band $B$, that is any semigroup in which the formula $xy=x$ holds, is an assembly in which distinct  idempotents do not commute (even if they still are a semigroup).

 In particular the multiplicative structure $B=\{a,b\}$, with $a^2=ab=a\ne b=ba=b^2$ is a non-commutative a band (and hence an assembly).

\end{example}

Let us see now a more complicated, but very natural, example of assembly. Before, let us recall that in the set $P(S)$ of all subsets of a semigroup $S$ one can define a multiplication of  $X,Y \in P(S)$ by the setwise product $XY=\{xy\ | x\in X , y\in Y\}$ and get a semigroup structure for $P(S)$, called the \emph{power-semigroup} \cite[I.7.5]{PetrichIntro}. 

\begin{example} The set $\cal A (G)$ of all cosets  $gN$ of all normal subgroups $N$ of a group $G$ is a subsemigroup of the power-semigroup $P(G)$ since  $g_1N_1=\{g_1\}N_1$ and for each $g_1,g_2\in G$ we have $$(g_1N_1)(g_2N_2)=(g_1g_2)N_1N_2,$$  where $N_1,N_2$ lie in the (semi)lattice i$n(G)$ of all normal subgroups of $G$, which is 
 a subsemigroup of   $\cal A (G)$. 

The functions \ $e(gN)=N$ and \ 
$s(gN)=g^{-1}N$ equip $\cal A (G)$ with a structure of  assembly, as it can be easily checked, which is non-commutative if $G$ is non-commutative. 
Then $E(G)=n(G)$ and any coset $gN$ belongs only to the group $G/N$ which  is, of course, a subsemigroup of 
$\cal A (G)$. Thus, the semigroup ${\cal A}(G)$ is the union of the factor groups $G/N$ (with their multiplicity):
$${\cal A}(G)=\bigcup\nolimits_{N\in n(G)}G/N$$ 
In such a structure idempotents do commute. Finally,  when $G$ is cyclic with order $p^n$, the assembly ${\cal A}(G)$ has order $1+p+\dots+p^n$ while the assembly $G\times n(G)$ has order $(n+1)p^n$.  
\end{example}

Now a crucial example: there are unions of groups which are not assemblies.

\begin{counterexample}\label{REES} Let $R$ be the semigroup $M(2,G,2,P)$ of so-called Rees $2\times 2$-matrices  where $G=\{1,-1\}$ is a multiplicative group of order $2$ and $P=\begin{psmallmatrix}1 & -1\\1 & 1\end{psmallmatrix}$. 
Then $R$ is a union of groups which is not an assembly.\end{counterexample}

\begin{proof} If we equip the set of $2\times2$ matrices over any field with characteristic $\ne2$ with the multiplication $*$ defined by $A*B=APB$ (where juxtaposition on the right-hand side means the usual row-by-column product), which is clearly associative, then the set $S$ consisting of the matrices:\ \ 
 $A=\begin{psmallmatrix}1 & 0\\ 0& 0\end{psmallmatrix}$,
 $B=\begin{psmallmatrix}0 & 1\\ 0& 0\end{psmallmatrix}$,
 $C=\begin{psmallmatrix}0 & 0\\ 1& 0\end{psmallmatrix}$,
 $D=\begin{psmallmatrix}0& 0\\ 0& 1\end{psmallmatrix}$ and their opposites $-A,-B,-C,-D$ is a semigroup (with respect to $*$) which is the union of the non-trivial groups $\{A,-A\}, \{B,-B\}, \{ -C, C\}, \{D,-D\}$ with neutral elements $A,B,-C,D$ respectively. Thus,  $(A_1)$ and $(A_2)$ hold. However $(A_3)$ fails since, on the one hand, from $B*C=A$ one has $e(B*C)=e(A)=A=A^2$, while on the other hand  $e(B)*e(C)=B*(-C)=-A$ is not even idempotent as $(-A)^2=A^2=A$.
\end{proof}

\begin{definition}\label{band}  Let $S$ be a semigroup.
If there exists $\varphi:S \to B$ a semigroup homomorphism where $B$  is a band (resp. semilattice), we say that $S$ is a \emph{band (resp. semilattice) of the subsemigroups} $S_{e}:=\varphi^{-1}\{e\}$ with $e\in B$. 
\end{definition}

Note that in the above circumstances we have  that $$S=\bigcup\nolimits_{e\in B} S_{e}$$ where $S_{e}$ is a subsemigroup since by $x,y\in S_{e}$ it follows $xy\in S_{e}$ since $\varphi(xy)=\varphi(x)\varphi(y)=ee=e$.

\medskip 
We are now able to characterize assemblies in terms of bands of groups.

\begin{theorem}\label{assemb}
A semigroup  is an assembly if and only if it is a band of groups.
\end{theorem}

\begin{proof} Let $(A,\cdot)$ be an assembly, then by Proposition~\ref{preAssembly} and $(A_3)$  the map $x\mapsto e(x)$ is an homomorphism whose image is a subsemigroup which is the band $E(A)$.  
	
Conversely, let $A$ be
a band of groups via the homomorphism  $\varphi: A\to B$ an homomorphism of semigroups where $B$ is a band. If  $\varphi$ is the canonical map $e(x)$ there is nothing left to show since we may apply Proposition \ref{preAssembly} to get that  \ref{assemblyneut} and  \ref{assemblysim} hold and note that \ref{assemblye(xy)} just means that the map $e$ is an homomorphism. 

To treat the general case, let us show that, up to an isomorphism, we can reduce to the case $\varphi=e$. 
Let us define $\psi: B\to E(A)$ where $\psi(b)$ is the identity element of the group $\varphi^{-1}\{b\}$. Let us show that $\psi$ is the inverse map of the restriction $\varphi_1:E(S)\to B$ of $\varphi$. In fact $\varphi_1(\psi(b))=\varphi(\psi(b))=b$ by definition of $\psi$. Moreover for each $\varepsilon\in E(S)$ we have that $\psi(\varphi_1(\varepsilon))$ is the identity element of the group 
$\varphi^{-1}(\varphi(\varepsilon))$. On the other hand, $\varepsilon\in \varphi^{-1}(\varphi(\varepsilon))$ and $\varepsilon$ is idempotent. Thus $\varepsilon$ is the identity element of $\varphi^{-1}(\varphi(\varepsilon))$,  hence $\psi(\varphi_1(\varepsilon))=\varepsilon$. 

Thus $\psi=\varphi_1^{-1}$ is an homomorphism, as $\varphi$ is by hypothesis. Then $e=\psi\varphi$ is an homomorphism as well, as wished.
\end{proof}

\begin{corollary}\label{assemb+} 
For a semigroup $A$ the following are equivalent:
\begin{enumerate}
\item $A$ is an assembly whose magnitudes commute.
\item $A$ is an assembly whose magnitudes are central. 
\item $A$ is a semilattice of groups, 
\end{enumerate}
\end{corollary}

\begin{proof} 
Observe that $(1)$ and $(2)$ are equivalent by Clifford's Lemma \cite[Lemma 3.1]{Clifford41}. If they hold, then $(3)$ holds since $e(x)$ is the wished homomorphism. Finally, $(3)$ implies $(1)$ via Theorem \ref{assemb}.
\end{proof}

\begin{example}
If $G$ is a group and $S$ a semilattice, then $S\times G$ is a possibly non-commutative semilattice of groups.
\end{example}

\section{Subassemblies}

In group theory it is possible to characterize the subgroups of a given
group $\left( G,\cdot\right) $ as any nonempty subset of $G$ which is closed
under multiplication and  inversion. A similar
characterization of subassemblies of a given assembly holds.

\begin{proposition}\label{Theoremsubassembly} 
If $(A,\cdot,s,e)$ is an assembly and $B$ is a non-empty subset of $A$, then the following are equivalent. 
\begin{enumerate}
 \item  $\forall x,y\in B\  \  x\cdot s(y)\in B$ 

\item   $(B,\cdot_B, s_B,e_b)$ is an assembly, where  $\cdot_B, s_B,e_b$ are the restrictions to $B$ of $\cdot,s,e,$ respectively.
\end{enumerate}
\noindent
If these condition hold, we say that $B$ is a \emph{subassembly} of $A$. \ Moreover, if $A_e$ is  a Clifford component of $A$ with $e\in B$, then $B_e=A_e\cap B$ is a Clifford component of $B$.
\end{proposition}

\begin{proof} 
It is clear that (2) implies (1). Assuming (1), we have to prove that $B$ is closed under the maps $\cdot,e,s$. If $b\in B$, then $e(b)=bs(b)\in B$, as wished. Moreover, $s(b)=e(s(b))s(b)= e(b)s(b)\in B$. Finally if $b_1\in B$, then $b_1b=b_1s(s(b))\in B$. Therefore (2) holds. 
\end{proof}

To prove that a structure is a subassembly of a given assembly becomes quite
simpler using the previous result. We illustrate this with some relevant examples.

\begin{example}
The following are subassemblies of $\left( \mathbb{E},+\right) .$

\begin{enumerate}
\item $\left( \mathbb{R},+\right) $, because $\mathbb{R\subset E}$ and $%
\left( \mathbb{R},+\right) $ is a group.

\item $B=\left \{ x+A:x\in \mathbb{R}\right \} ,$ where $A$ is a given
neutrix. We have $A\in B$ and $B\subseteq \mathbb{E}$. If $\alpha =a+A$, $%
\beta =b+A\in B$ then $\alpha -\beta =\left( a+A\right) -\left( b+A\right)
=\left( a-b\right) +A\in B$.

\item $(\mathcal{N},+),$ where $\mathcal{N}$ is the class of all neutrices.
Note that the class of all neutrices is nonempty because $0\in \mathcal{N}$ and the
difference of two neutrices is equal to the larger of the two.

\item $(\mathbb{E}\backslash \mathbb{R},+).$ Clearly $\oslash \in \mathbb{E}%
\backslash \mathbb{R}$, hence $\mathbb{E}\backslash \mathbb{R}$ is nonempty.
Let $x=a+A$, $y=b+B\in \mathbb{E}\backslash \mathbb{R}$. Then $x-y=\left(
a-b\right) +\max \left( A,B\right) \in \mathbb{E}\backslash \mathbb{R}$.

\item $(A_{\rho },+)$, where $\rho \in \mathbb{R}$ and $A_{\rho }=
\{
x\in \mathbb{E}:x\subseteq \bigcup \nolimits_{\st(n)}\left[ -\rho
^{n},\rho ^{n}\right]  \} .$ Clearly $\emptyset \neq A_{\rho
}\subseteq \mathbb{E}$. Let $x,y\in A_{\rho }$. Then there are standard $m,n$
such that $x\subseteq \left[ -\rho ^{m},\rho ^{m}\right] $ and $y\subseteq %
\left[ -\rho ^{n},\rho ^{n}\right] .$ Let $p=\max \left \{ m,n\right \} $.
Then $\left \vert x-y\right \vert \leq 2\max \left \{ x,y\right \} \leq
2\rho ^{p}\leq \rho ^{p+1}$.

\item Let $(A,+)$ and $(B, \cdot)$ be assemblies. Let $(G,+)$ be a subassembly of $(A,+)$ and $(H, \cdot)$ be a subassembly of $(B,\cdot)$. Then $(G\times H, \ast)$ is a subassembly of $(A \times B, \ast)$.

\item  If $H$ is any subgroup of a group $G$, then ${\cal A}(H)$ is a subassembly of ${\cal A}(G)$. 
\end{enumerate}
\end{example}

An important difference between assemblies and groups is that subassemblies do not need
to contain all a universal neutral element, allowing both $(\mathbb{E}\backslash \mathbb{R},+)$ and $%
\left( \mathbb{R},+\right) $ to be subassemblies of $\left( \mathbb{E}%
,+\right) $. This fact shows that, unlike what happens with groups, {\em it is possible for the intersection of two subassemblies of a given assembly to be empty}. Moreover, for $\left( B,\cdot\right) ,\left( C,\cdot\right) $ subassemblies of an assembly $\left( A,\cdot\right)$ it may happen that $B\cup C$ is a subassembly of $A$ and both $B\nsubseteq C$ and $C\nsubseteq B$. However, the following holds.

\begin{proposition}\label{P:Intandsumassemb}
Let $B,C $ be subassemblies of an assembly $A$. Then  $B\cap C$ is either empty or  a subassembly of A. Moreover, if $A$ is commutative, the set $B\cdot C$ is a subassembly of $A$, where the product is meant to be defined in the power-semigrop $P(A)$.
\end{proposition}

\begin{proof}
Suppose that $B\cap C$ is nonempty. Let $x,y\in B\cap C$. Then, because $B$
and $C$ are assemblies, $x\cdot y^{-1}\in B$ and $x\cdot y^{-1}\in C$ and then $x\cdot y^{-1}\in B\cap C$.
Hence $\left( B\cap C,\cdot\right) $ is a subassembly of $\left( A,\cdot\right) $,
by Proposition~\ref{Theoremsubassembly}.

Suppose now that $x,y\in B\cdot C$. Then there are $u,v\in B$ and $r,t\in C$,
such that $x=u\cdot r$ and $y=v \cdot t$. Because $B$ and $C$ are assemblies, $u\cdot v^{-1}\in B$\ ,\ $r \cdot t^{-1}\in C$ and then $
x\cdot y^{-1} =(u\cdot r)\cdot  ( v\cdot t)^{-1} =(u\cdot v^{-1})\cdot (r\cdot t^{-1})\in B\cdot C$, by formulas (F) and because $A$ is commutative.
Hence  $\left( B\cdot C,\cdot\right) $ is a
subassembly of $\left( A,\cdot\right) $, by Proposition~\ref{Theoremsubassembly}.
\end{proof}

\begin{example}
By Proposition~\ref{P:Intandsumassemb} we have that  $A=\left \{ x+N:x\in \mathbb{Z},N\in \mathcal{N}%
\right \} ,$ $B=\left \{ x+\oslash :x\in \mathbb{Q}\right \} $ and $%
C=\left
\{ x\in \mathbb{Z}:x\text{ is limited}\right \} $ are assemblies
because $A=\mathbb{Z+}\mathcal{N}$, $B=\mathbb{Q+}\oslash $ and $C=\mathbb{%
Z\cap }\pounds $.
\end{example}

\begin{proposition}
The subset $Z(A)$ of elements of an assembly $A$ commuting with all elements of $A$ is a subassembly of $A$ that we call the centre of $A$. 
\end{proposition}

\begin{proof} If $z,z'\in Z(A)$ it is trivial that $zz'\in Z(A)$. Let us prove that $z^{-1}\in Z(A)$. For each $a\in A$, by formulas $(F)$ we have $az^{-1}=((az^{-1})^{-1})^{-1}=(za^{-1})^{-1}=(a^{-1}z)^{-1}=z^{-1}a$ and we may apply Proposition \ref{Theoremsubassembly}.
\end{proof}


\section{Homomorphisms}\label{Section Homomorphisms2}

An homorphism $\varphi$ between two assemblies $(A,\cdot_A, s_A,e_A)$ 
and $(B,\cdot_B, s_B,e_B)$ is expected to be a map $\varphi:A\to B$ which respects the $3$ given  operations.  However, in the case under consideration, the request is so easily fulfilled that we can proceed  even with a slight abuse of notation as in the next proposition.

\begin{proposition}\label{Prophom} Let $A$ and $B$ be assemblies.\\ 
If  $\varphi:A\to B$ is a semigroup homomorphism,  i.e.\  if  $\varphi(xy)=\varphi(x)\varphi(y)$, for all $ x,y\in A$, then for each $x$ in $A$
\begin{enumerate}
\item \label{Prophom2}
$\varphi (x^0)=\varphi (x)^0,$  
\item \label{Prophom1}
$\varphi (x^{-1})=\varphi (x)^{-1}$. 
\end{enumerate}
\end{proposition}

\begin{proof}
We have $\varphi(x)\varphi(x^0)=\varphi(xx^0)=\varphi(x)=\varphi(x^0x)=\varphi(x^0)\varphi(x)$, hence by the uniqueness  of $\varphi(x)^0$ we deduce $\varphi(x)^0= \varphi(x^0)$. The second part follows in a similar way since 
$\varphi(x)\varphi(x^{-1})=\varphi(xx^{-1})=\varphi(x^0)=\varphi(x)^0=\varphi(x^0)=\varphi(x^{-1}x)=\varphi(x^{-1})\varphi(x)$.
\end{proof}
Thus the homomorphic image of a magnitude is a
magnitude and the homomorphic image of the inverse of a given
element is the inverse of the homomorphic image of that same element. These
properties generalize similar properties for group homomorphisms.

\begin{example}\label{E: hom}

The following are assembly homomorphisms (sometimes in additive notation):

\begin{enumerate}
\item All group homomorphisms, because every group is an assembly.

\item The identity map $f(x)=x$ and the map $e(x)=x^0$ are assembly homomorphisms.

\item Let $A$ be a neutrix. Then $f:(\mathbb{E},\mathbb{+)}\to (%
\mathbb{E},\mathbb{+)}$ such that $f(x)=x+A$ is an homomorphism. In fact, if 
$x,y\in \mathbb{E},$ 
\begin{equation*}
f(x+y)=(x+y)+A=x+y+A+A=(x+A)+(y+A)=f(x)+f(y).
\end{equation*}

\item The function $f:(\mathbb{E},\mathbb{+)}\to (\mathbb{E},%
\mathbb{+)}$ such that $f(x)=\omega x$ for some $\omega \simeq +\infty $ is
an homomorphism. Let $x,y\in \mathbb{E}$. Then, using \cite[Lemma~5.12]{DinisBerg(11)},%
\begin{equation*}
f(x+y)=\omega (x+y)=\omega x+\omega y=f(x)+f(y)\text{.}
\end{equation*}

\item \label{exemplo hom zerobarrax}The function $f:(\mathcal{N},\mathbb{+)}%
\to (\mathcal{N},\mathbb{+)}$, $f(x)=\oslash x$, where $\oslash $
is the external set of infinitesimal numbers. Let $x,y\in \mathcal{N}$.
Using \cite[Corollary~5.10]{DinisBerg(11)},
\begin{equation*}
f(x+y)=\oslash (x+y)=\oslash x+\oslash y=f(x)+f(y)\text{.}
\end{equation*}

\item 
The function $f:(\mathcal{N},\mathbb{+)}%
\to (\mathbb{E}\backslash \left \{ 0\right \} ,\mathbb{\cdot )}$
such that $f(x)=\exp _{S}\left( x\right) \equiv \left[ -e^{x},e^{x}\right] $ (see \cite[Def. 1.4.2]{Koudjeti})
. Let $A,B\in \mathcal{N}$. Then \vskip-10mm %
\begin{eqnarray*}
\exp _{S}\left( A+B\right) &=&[\left( -e^{A}\right) e^{B},\left(
e^{A}\right) e^{B}]=[-e^{A},e^{A}]e^{B} \\
&=&[-e^{A},e^{A}][-e^{B},e^{B}]=\exp _{S}\left( A\right) \exp _{S}\left(
B\right).
\end{eqnarray*}

\item  If $G$ is a group, the function $$(g,N)\in G\times n(G)\mapsto gN\in \cal A(G)$$ is a possibly non-injective  epimorphism of assemblies,  where  from $g_1N_1=g_2N_2$ it follows if $N_1=N_2$, by applying the function $e$.

\end{enumerate}
\end{example}

\begin{counterexample}
Obvious examples of functions which are not homomorphisms are nonlinear
functions. Consider for instance the function $f:(\mathbb{E},\mathbb{+)}%
\to (\mathbb{E},\mathbb{+)}$ such that $f(x)=x^{2}$. In fact, if 
$x=-1+\oslash $ and $y=1+\oslash $ then 
\begin{equation*}
f(x+y)=f\left( \oslash \right) =\oslash ^{2}
\end{equation*}%
and 
\begin{equation*}
f\left( x\right) +f\left( y\right) =\left( 1+\oslash \right) ^{2}+\left(
-1+\oslash \right) ^{2}=\left( 1+\oslash \right) +\left( 1+\oslash \right)
=2+\oslash \text{.}
\end{equation*}

However there are also functions which may appear to be linear but are
really not. As such one may not extend Example \ref{E: hom}.\ref{exemplo hom zerobarrax}
to the whole of $\mathbb{E}$: 
\begin{equation*}
\oslash (1-1)=0,
\end{equation*}%
while 
\begin{equation*}
\oslash 1-\oslash 1=\oslash .
\end{equation*}
\end{counterexample}

\begin{proposition}
Let $\varphi :A\to B$ be an assembly homomorphism.
Then $\varphi (A)$ is a subassembly of $B.$
\end{proposition}

\begin{proof} Apply Propositions \ref{Theoremsubassembly} and \ref{Prophom}.
\end{proof}

Thus, in studing assembly homomorphisms, there is not much loss of generality in assuming that these are onto. Furthermore, since -via the Clifford decomposition- any assembly $A$ may be partitioned into disjoint groups and the homomorphic image of a group is likewise a group, one may regard any assembly homomorphism 
$$A=  \bigcup_{e\in E(A)}A_e\ \  {\xrightarrow{\varphi}}\ \  B= \bigcup_{\varepsilon\in E(B)}B_\varepsilon\phantom{.......}$$ \vskip-3mm
\noindent as a disjoint union of group homomorhisms\phantom{......}$A_e\ \ {\xrightarrow{\varphi_e} }\ \ B_{\varphi(e)}$ \ .\ \  Then, for  $e\in E(A)$, we   define   $\varphi_e$  to be  the $e$-th component of $\varphi$ and we have:

\begin{proposition}
An assembly homomorphism is into (resp. onto) if and only if all components of its are  into (resp. onto).\hfill $ \square$  
\end{proposition}

Therefore, if we denote by $\Ker({\varphi })$   the usual kernel  of $\varphi$ we have that it is the union of the kernels of its components $\varphi_e$. By the above this is a subassembly and we have: 
$$\Ker({\varphi })= \bigcup\nolimits_{e\in E(A)} 
 \left \{ x\in A:\varphi (x)=\varphi (e)\right \} =  
\bigcup\nolimits_{e\in E(A)} ker_{\varphi_e}=\varphi^{-1}(\varphi(e(A))\supseteq E(A) $$

\begin{corollary} An homomorphism of assemblies $\varphi$ is injective if and only if \ $\Ker_{\varphi }=E(A)$, i.e. its kernel coincides with the set of idempotents of $S$.\hfill $ \square$
\end{corollary}

\begin{example} If $A$ is an assembly and its semilattice $E(A)$ of idempotents has maximum $m$, then there are no non-trivial homomorphisms $\varphi:A\to G$ to any group $G$. This holds in particular for the assembly of cosets $A={\cal A}(G)$. 
\end{example} 

\begin{proof} By Proposition \ref{Prophom}, the element $\varphi(m)$ must be idempotent, but in $G$ there is only one idempotent:  its unique neutral element $1_G$. Then  
$\varphi(a)= \varphi(a)\cdot 1_G=\varphi(a)\varphi(m)=\varphi(am)= \varphi(m)=1_G$,  for all $a\in A$.
\end{proof}

\subsection*{Acknowledgments}

 The second author acknowledges the support of FCT - Funda\c{c}\~ao para a Ci\^{e}ncia e Tecnologia under the projects: UIDP/04561/2020 and UIDP/04674/2020, and the research centers CMAFcIO -- Centro de Matem\'{a}tica, Aplica\c{c}\~{o}es Fundamentais e Investiga\c{c}\~{a}o Operacional and CIMA -- Centro de Investigação em Matemática e Aplicações.


\end{document}